\documentclass[preprint,review]{elsarticle}

\usepackage[hidelinks]{hyperref}
\usepackage{lineno}
\usepackage[margin=4cm]{geometry}

\let\today\relax
\makeatletter
\def\ps@pprintTitle{%
	\let\@oddhead\@empty
	\let\@evenhead\@empty
	\def\@oddfoot{\footnotesize\itshape
		{Submitted preprint} \hfill\today}%
	\let\@evenfoot\@oddfoot
}
\makeatother

\usepackage{amsmath}
\usepackage{amssymb}
\usepackage{amsthm}
\usepackage{dsfont}

\usepackage{optidef}
\usepackage{url}
\usepackage{color}
\usepackage{multirow}
\usepackage{subfig}
\usepackage{eurosym}
\newtheorem{case}{Case}
\newtheorem{theorem}{Theorem}
\newtheorem{lemma}[theorem]{Lemma}

\newtheorem{example}[theorem]{Example}
\newtheorem{remark}[theorem]{Remark}

\newtheorem{proposition}[theorem]{Proposition}

\usepackage{float}
\usepackage{array}
\usepackage{multirow}
\usepackage{textcomp}
\newcolumntype{L}{>{\centering\arraybackslash}m{1.7cm}}
\newcolumntype{R}{>{\centering\arraybackslash}m{3cm}}
\usepackage{tikz}
\usetikzlibrary{arrows}
\usepackage{makecell}
\usepackage{pgfplots}
\DeclarePairedDelimiter\abs{\lvert}{\rvert}
\usepackage{afterpage}
\usepackage[normalem]{ulem}
\usepackage[inline,shortlabels]{enumitem}
\usepackage{accents}
\newcommand{\ubar}[1]{\underaccent{\bar}{#1}}
\usepackage{acro}

\DeclareAcronym{mip}{
	short = MIP ,
	long  = Mixed-Integer Programming ,
	sort  = M ,
}
\DeclareAcronym{der}{
	short = DER ,
	long  = Distributed Energy Resources ,
	sort  = D ,
}
\DeclareAcronym{rtbm}{
	short = RTBM ,
	long  = Real-Time Balancing Market ,
	sort  = R ,
}
\DeclareAcronym{ads}{
	short = ADS ,
	long  = Active Demand and Supply ,
	sort  = A ,
}
\DeclareAcronym{brp}{
	short = BRP ,
	long  = Balance Responsible Party ,
	sort  = B ,
}
\DeclareAcronym{tso}{
	short = TSO ,
	long  = Transmission System Operator ,
	sort  = T ,
}
\DeclareAcronym{hp}{
	short = HP ,
	long  = Heat Pump ,
	sort  = H ,
}
\DeclareAcronym{mchp}{
	short = mCHP ,
	long  = Micro Combined Heat and Power  ,
	sort  = m ,
}
\DeclareAcronym{mpec}{
	short = MPEC ,
	long  = Mathematical Programming with Equilibrium Constraints  ,
	sort  = M ,
}

\biboptions{sort&compress}
\begin{document}
	
	\begin{frontmatter}
		
		\title{A Real-Time Balancing  Market Optimization with Personalized Prices: From Bilevel to Convex}

		\author[BI]{Koorosh Shomalzadeh\corref{mycorrespondingauthor}}
		\cortext[mycorrespondingauthor]{Corresponding author}
		\ead{k.shomalzadeh@rug.nl}
		
		\author[ENTEG]{Jacquelien M. A. Scherpen}
		\ead{j.m.a.scherpen@rug.nl}
		
		\author[BI]{M. Kanat Camlibel}
		\ead{m.k.camlibel@rug.nl}

		\address[BI]{Jan C. Willems Center for Systems and
			Control, Bernoulli Institute for
			Mathematics, Computer Science and Artificial Intelligence, Faculty of Science and Engineering, University of Groningen, Nijenborgh 9, 9747 AG, Groningen,
			The Netherlands}
		\address[ENTEG]{Jan C. Willems Center for Systems and
			Control, Engineering and Technology Institute Groningen, Faculty of Science and Engineering, University of Groningen, Nijenborgh 4, 9747 AG, Groningen,
			The Netherlands}
		\begin{abstract}
			This paper studies the static economic optimization problem of a system with a single aggregator and multiple prosumers  in a \ac*{rtbm}. 
			The aggregator, as the agent responsible for  portfolio balancing, needs to minimize the cost for imbalance satisfaction in real-time by proposing a set of optimal personalized prices to the prosumers. On the other hand, the prosumers, as  price taker and self-interested agents,  want to maximize their profit by changing their supplies or demands and providing flexibility based on the proposed personalized prices.
			We model this problem as a bilevel optimization problem. 
			We first show that the optimal solution of this bilevel optimization problem can be found by solving an equivalent convex problem. In contrast to the state-of-the-art \ac*{mip}-based approach to solve bilevel problems, this convex equivalent has very low computation time and is  
			appropriate for real-time applications.   
			Next, we compare the optimal solutions of the proposed personalized scheme and a uniform pricing scheme.  
			We prove that, under the  personalized pricing scheme, more prosumers contribute to the  \acs{rtbm} and the aggregator's cost is less.
			Finally, we verify the analytical results of this work by means of numerical case studies and simulations.  
		\end{abstract}
		

	\end{frontmatter}

	\section{Introduction}
	In  recent years, the increase in the penetration of  \ac{der}s  at the demand side has  drastically changed  the structure of our power system.
	As a result, the old passive households, which only consumed energy, found a more active role with the help of the demand side generation. The new term \textit{prosumer}  was introduced in the energy community to represent this transition for households \cite{parag2016electricity}.
	\par
	The emergence of prosumers calls for  a new real-time market structure in contrast to the existing day ahead  and  intraday markets. Since output power of many \ac{der}s is volatile due to their intrinsic environmental dependency, planning for supply and demand matching needs to be done as close as possible to real-time to keep the system stable and economically efficient. Therefore, a \acf{rtbm}  \cite{pineda2013using} that incorporates available unused capacity of prosumers' controllable \ac{der}s and flexible loads, which together we denote here as controllable \ac{ads} units, should be developed to address  the supply volatility by incentivizing prosumers.
	\par
	Currently, there is only an ex-post financial settlement procedure in the Netherlands and most of Europe, and  no  actual or physical real-time balancing  occurs \cite{wang2015review}.   
	Communication infrastructure  in the new paradigm of smart grid \cite{conejo2010real} facilitates the participation  of  the prosumers with  controllable \ac{ads} units in an \ac{rtbm}.
	Moreover, to prevent direct interaction of the prosumers with  higher level agents in the market and aggregate them, a market participant, the \textit{aggregator},  has been introduced \cite{gkatzikis2013role}.
	The aggregators  have  different roles in different market structures.
	\par
	The goal of an aggregator in an \ac{rtbm} is to optimize its operational costs for balancing 
	by incentivizing  the prosumers to utilize their unused assets.  
	There are many approaches which  an aggregator can employ to steer its associated prosumers to an optimal  operation  point \cite{vardakas2014survey}. One of the most popular approaches is to consider the aggregator as a leader, who can  anticipate the reaction of the prosumers, proposes some prices to the following prosumers such that their reactions would be optimal for the aggregator.
	This price incentive oriented setup falls into the category of \textit{bilevel optimization problems} \cite{colson2007overview}    and \textit{Stackelberg games} \cite{von2010market}, where the lower level problems and the upper level problem  are the problems related to the prosumers and the aggregator, respectively.
	\par 
	The bilevel and Stackelberg game modeling of the aggregator and prosumers' interactions have been studied extensively in the literature \cite{zugno2013bilevel,meng2016bilevel,tushar2016price,yang2018model,yang2017framework,tushar2012economics,hobbs2000strategic}. 
	Two different pricing schemes have been proposed to incentivize prosumers in the aforementioned studies.  The \textit{uniform pricing} scheme is an incentivization scheme where the aggregator proposes the same price to all of the prosumers \cite{zugno2013bilevel,meng2016bilevel}. In the other pricing, i.e., the \textit{personalized pricing} scheme, the aggregator proposes a unique price to each prosumer in order to reach its goal \cite{tushar2014prioritizing,yang2018model,yang2017framework}. While these two pricing schemes have been considered in different works interchangeably and it is argued that the personalized pricing scheme has some benefits over uniform pricing scheme, there exists no research which provides rigorous mathematical proofs on the differences between these two schemes. 
	\par
	Moreover, the state-of-the-art approach to solve these types of bilevel optimization problems is to solve them as  \ac{mip}s \cite{zugno2013bilevel,li2018participation,wang2017strategic}. However,  
	implementing the mentioned setup in  real-time requires very fast computations. The time intervals  for a real-time balancing market can often be as low as $5$ minutes \cite{vlachos2013demand}. Therefore, the solution for each interval has to be computed and executed within seconds or even less.  
	While papers  like \cite{ghamkhari2016strategic} have studied  the computational efficiency of the bilevel optimization correspond to generating firms strategic offering by introducing a convex relaxation, to the best of our knowledge,   no study addressed the computation time for the prosumers/aggregator setup with personalized prices {for a high number of prosumers. It should be noted that, although the algorithms in \cite{tushar2012economics} and \cite{tushar2014prioritizing} are distributed, their efficiency are not guaranteed for large problems and real-time applications.} 
	\par 
	In contrast to the above works, here we stick to a simple model for the aggregator and prosumers interaction with personalized pricing scheme to analyse   the corresponding bilevel optimization problem in a fundamental and tractable mathematical way. Although our model is simple, we  keep the essence of these  market models and most of the results in this paper can be generalized to more complicated and realistic models.
	\par
	\textit{Contributions:}
	We present a bilevel optimization problem to model the interactions between  self-interested  aggregator and prosumers in an \ac{rtbm}. A personalized pricing scheme by the aggregator is proposed  to incentivize the prosumers to participate in this market. 
	Bilevel problems, in general, are non-convex \cite{luo1996mathematical}.
	We first prove that the global optimal solution of this bilevel optimization problem can be found by solving a convex equivalent problem. 
	This convex equivalent formulation has two main advantages. On the one hand, it guarantees global optimality. On the other hand, a  convex formulation is attractive in real-time applications with high number of prosumers since the other approaches to solve bilevel optimization problems (e.g., \ac{mip}-based approach) are not computationally efficient.
	Afterwards,  we compare the optimal solution of the proposed model with personalized prices to a uniform pricing scheme. We prove that the personalized pricing scheme leads to a less cost for the aggregator and under this pricing scheme more prosumers contribute to the balancing market. 
	\textcolor{black}{Preliminary results of this work are partially presented in the extended abstract \cite{shomalzadeh2020solution}. In contrast to the abstract, this paper considers a more general model for the prosumer and provides theoretical proofs for the results. Also, in this paper we compare uniform and personalized pricing schemes in different aspects.}
	\par
	The paper is organized as follows. Section~\ref{sec:pf} explains the prosumers/aggregator interaction model in a real-time balancing market and introduces the bilevel problem. In Section~\ref{sec:ce}, we show that the bilevel optimization problem is equivalent to a certain convex problem.  
	The analytical comparison of the optimal solution of the proposed personalized pricing scheme and a uniform pricing  scheme is presented in Section~\ref{sec:vs}.
	The efficiency of the proposed method is illustrated by means of simulations in Section~\ref{sec:sim}.  Section~\ref{sec:con} concludes the paper.
	\textcolor{black}{The proofs of some theoretical results are presented in \ref{app}.}
	\section{Problem formulation} \label{sec:pf}
	In this section, we formulate the static bilevel economic optimization problem of  an aggregator and its portfolio for participation in an \ac{rtbm}. 
	{While this paper is devoted to investigate a single time-step, the proposed scheme can also be applied for dynamic cases with multiple time-steps. The general structure of this market is as follows.} 
	Each aggregator has a set of prosumers under  contract and each prosumer is on a contract with only one aggregator.  There are many types of aggregators in an electricity market. In this paper, we consider a commercial aggregator which also acts as a  \ac{brp} \cite{ding2013real}.  Therefore, the aggregator here is also responsible for balancing its portfolio. To do so, the aggregator receives a real-time price from the \ac{tso}, who usually has the highest role in the market hierarchy,  and incentivizes the prosumers with personalized prices to supply or consume more or less based on that.
	The change in each prosumer electrical energy  supply or demand  in a time interval is referred as \textit{flexibility}.
	Next, we explain the problem setting and market structure in detail.
	
	Prosumers are equipped with various kinds of  \ac{ads} units. They consist of two prominent categories, namely controllable and uncontrollable units.  \ac{mchp} units and \ac{hp} units are examples of controllable active supply and demand units of electricity, respectively.  Output generation of units such as solar cells and wind turbines is dependent on environmental conditions. Thus these  are uncontrollable supply units. \textcolor{black}{ Throughout this paper, we assume that each prosumer has a modular \ac{mchp} and  \ac{hp}  as its controllable \ac{ads} units and it might have a solar panel or  wind turbine as an uncontrollable one.} Each prosumer heat demand is also assumed  to be flexible by considering a loss of comfort factor, that is, it is willing to consume more or less heat if its loss of comfort is compensated by the aggregator. Since heat is an output for both \ac{mchp} and \ac{hp}, prosumers are able to alter their controllable \ac{ads} units output level to participate in the balancing market. 
	\par
	Due to the uncertain nature and volatility  of both the  uncontrollable \ac{der}s   and the prosumers demand, there could be a mismatch between the pre-planned  supply and demand schedules in the real-time. To balance this mismatch  and to participate in the \ac{rtbm}, the aggregator  incentivizes the prosumers with personalized  prices \cite{yang2017decision} in a centralized way to consume or  supply more energy using their controllable \ac{ads} units.  
	Before providing a precise mathematical formulation, we elaborate on some technical notions.
	\par
	The aggregator is in \textit{up-regulation} if its prosumers' demand is lower than its  supply. Similarly, the aggregator is in \textit{down-regulation} if the demand is higher than the supply for its  prosumers. 
	Likewise, the \ac{tso} is in \textit{up-regulation} if the total system demand is lower than the total system generation. Otherwise, it is in \textit{down-regulation}. 
	Based on these definitions, we distinguish the following four cases:
	\begin{case} \label{cs:1}
		The aggregator and the \ac{tso} both are in up-regulation:  \normalfont{The aggregator needs to pay the \ac{tso} to take care of its excess supply or it can incentivize the prosumers with \ac{mchp} to generate less and the prosumers with \ac{hp} to consume more.}
	\end{case}
	\begin{case} \label{cs:2}
		The aggregator is in up-regulation and the \ac{tso} is in down-regulation:  \normalfont{The \ac{tso} pays the aggregator for its excess supply.}
	\end{case}
	\begin{case} \label{cs:3}
		The aggregator and the \ac{tso} both are in down-regulation: \normalfont{The aggregator needs to pay the \ac{tso} to provide supply or it can incentivize the prosumers with \ac{mchp} to generate more and the prosumers with \ac{hp} to consume less.}
	\end{case}
	\begin{case} \label{cs:4}
		The aggregator is in down-regulation and the \ac{tso} is in up-regulation: \normalfont{The \ac{tso} pays the aggregator to consume more.}
	\end{case}
	In both Case~\ref{cs:2} and  Case~\ref{cs:4} the solution for the optimal strategy  of the aggregator is trivial: sell the requested flexibility to the \ac{tso}. However, in Case~\ref{cs:1} and Case~\ref{cs:3} the aggregator needs to find a trade-off between the possible options for the optimal strategy. In the following subsection, we focus on modeling  Case~\ref{cs:1} and Case~\ref{cs:3} as a bilevel optimization problem.
	\subsection{The prosumers/aggregator model}
	We consider both the aggregator and the prosumer as self-interest agents. The aggregator tries to minimize its cost to settle the imbalance and the prosumer's goal is to maximize its revenue and minimize its cost and discomfort by altering its demand or supply given the personalized price proposed by the aggregator. 
	\par
	\textcolor{black}{
	We consider one aggregator and $n$ prosumers each has one \ac{hp} and \ac{mchp}.
	For all $i\in N=\{1,2,\dots,n\}$, we denote the proposed personalized price by the aggregator to the $i$th prosumer by $x_i$ and the prosumer $i$'s \ac{hp} and \ac{mchp} optimal flexibility response by $y_{i1}$ and $y_{i2}$, respectively. 
	Accordingly, we reserve the subscripts $i1$ and $i2$ to denote the parameters of the $i$th prosumer's \ac{hp} and \ac{mchp}, respectively.
	To model both Case~\ref{cs:1} and Case~\ref{cs:3}, we employ the following optimization problem for each prosumer:
	\begin{subequations}  \label{prob00}
		\begin{alignat}{2} \label{prob001}
			& \underset{y_{i1},y_{i2}}{\mathrm{max}} \quad && x_i(y_{i1}+y_{i2})-(f_i(y_{i1},y_{i2})+b_{i1}y_{i1}+b_{i2}y_{i2})\\ \label{prob002}
			&\mathrm{subject \ to} \quad && 0 \le y_{i1} \le m_{i1},\\
			&&& 0 \le y_{i2} \le m_{i2},
		\end{alignat} 
	\end{subequations}
	where $m_{i1},m_{i2}>0$ are the maximum available flexibility, $b_{i1}$ and $b_{i2}$ are the prices of providing flexibility and 
	{$f_i(y_{i1},y_{i2})$ is the discomfort function for prosumer $i$. In this work, we consider $f_i(y_{i1},y_{i2})=\frac{1}{2}(\sqrt{a_{i1}}y_{i1}-\sqrt{a_{i2}}y_{i2})^2$ where the  parameters  $\sqrt{a_{i1}},\sqrt{a_{i2}} $ translate the flexibility provision to heat increase/decrease  \cite{deng2014residential}.
	Note that in both Case~\ref{cs:1} and Case~\ref{cs:3}, the \ac{hp} and \ac{mchp}'s heat outputs due to flexibility provision change in the opposite direction. For instance, in Case~\ref{cs:1}, the aggregator rewards the prosumer to increase its HP consumption and decrease its mCHP generation. This leads to more heat generation for the HP and less for the mCHP. Therefore, we have employed minus sign in the discomfort function definition.}
	Next, we elaborate further on the model and parameters.}
	\par
	\textcolor{black}{
	In \eqref{prob001}, the first term corresponds to the received payment by the prosumer $i$ from the aggregator. The second term models the discomfort of the  prosumer $i$ for providing flexibility $y_{i1}$ and $y_{i2}$. Finally, the last two terms capture the amount  prosumer $i$ can save or the cost it should pay with respect to the intraday market plannings for providing flexibility $y_{i1}$ and $y_{i2}$.}
	\par
	\textcolor{black}{
		The parameter $b_{i1}$ for the  prosumer's \ac{hp} in both the aggregator up-regulation (Case~\ref{cs:1}) and down-regulation (Case~\ref{cs:3})  is  as follows:
		\begin{equation*}
			b_{i1}=\begin{cases*}
				\pi_e & \textrm{if aggregator  in up-regulation,} \\
				-\pi_e & \textrm{if  aggregator  in down-regulation,}
			\end{cases*}
		\end{equation*}
		Likewise, for the prosumer's  \ac{mchp} this parameter is defined as follows:
		\begin{equation*}
			b_{i2}=\begin{cases*}
				-c_i \pi_g  & \textrm{if  aggregator  in up-regulation,} \\
				c_i \pi_g & \textrm{if  aggregator  in down-regulation,}
			\end{cases*}
		\end{equation*}
		where $c_i$  is dependent on the \ac{mchp} technology of the prosumer $i$ and is given by 
		\begin{equation*}
			c_i=\frac{\textrm{nominal input power }}{\textrm{nominal electricity output power}}\cdot
		\end{equation*}
		and  $\pi_e \ge 0$ and $\pi_g \ge 0$ are fixed electricity and gas prices charged by the electricity and gas suppliers, respectively.}
	\par
	\textcolor{black}{
	Further, we define the maximum available flexibility $m_{i1}$ and $m_{i2}$ as follows.  
	For prosumer $i$, let $P_{i1}$ and $P_{i2}$ denote  the input electrical power to an \ac{hp} device and  the output electrical power of an \ac{mchp} device, respectively. Also, let $P_{i1}^{\mathrm{max}}$ and $P_{i2}^{\mathrm{max}}$  denote the maximum electrical power for prosumer $i$'s \ac{ads} devices.
	Then, the maximum available flexibility of the prosumer $i$'s \ac{hp}  is given by
	\begin{equation*}
		m_{i1}=\begin{cases*}
			(P_{i1}^{\mathrm{max}}-P_{i1})\Delta t & \text{if  aggregator  in  up-regulation,} \\
			P_{i1}\Delta t & \text{if  aggregator  in  down-regulation,} \\
		\end{cases*}
	\end{equation*}
	where $\Delta t$ is the duration of each time step for the \ac{rtbm} and assumed to be equal to $300$ seconds in this paper. Similarly, we define $m_{i2}$ for a prosumer's \ac{mchp} as follows:
	\begin{equation*}
		m_{i2}=\begin{cases*}
			P_{i2}\Delta t & \text{if  aggregator  in  up-regulation,} \\
			(P_{i2}^{\mathrm{max}}-P_{i2})\Delta t & \text{if  aggregator  in  down-regulation.} \\
		\end{cases*}
	\end{equation*} }
	{
		\begin{figure}[t!] 
			\centering
			\begin{tikzpicture}[scale=0.70]
				\tikzset{vertex/.style = {shape=circle,draw, minimum size=1cm}}
				\tikzset{edge/.style = {->,> =  triangle 45}}
				\tikzset{dedge/.style = {dashed,->,> =  triangle 45}}
				\node [shape=circle,draw, minimum size=2cm] (TSO) at (4,8) {TSO};
				\node [shape=rectangle,draw, minimum width=3cm,minimum height=2cm] (Agg) at (4,4) {Aggregator};
				\node[vertex] (pro1) at  (0,0) { Pro. $1$};
				\node[vertex] (pro2) at  (8,0) { Pro. $n$};
				\node [shape=circle, minimum size=1cm] (pron) at (4,0) {\huge \dots};
				\draw[dedge] (Agg) to[bend right] node[label=above:$x_1$]{} (pro1);
				\draw[dedge] (Agg) to[bend left] node[label=above:$x_{n}$]{} (pro2);
				\draw[edge] (pro1) to[bend right] node[label=below:$y_1$]{} (Agg);
				\draw[edge] (pro2) to[bend left] node[label=below:$y_{n}$]{} (Agg);
				\draw[dedge] (TSO) to[bend right] node[label=left:$p$]{} (Agg);
				\draw[edge] (TSO) to[bend left] node[label=right:$f-\sum_{i}y_i$]{} (Agg);
				\draw[thick, solid] (7,7.5) rectangle (11,9.5);
				\draw[solid] (7.5,8) -- (8,8) node[anchor= west] {\footnotesize Flexibility flow };
				\draw[dashed] (7.5,9) -- (8,9) node[anchor= west] {\footnotesize Price signal};
			\end{tikzpicture}
			\caption{A general overview of interactions for the aggregator, the prosumers and the \ac{tso} in  the \ac{rtbm}.  }
			\label{fig:interactions}
	\end{figure}
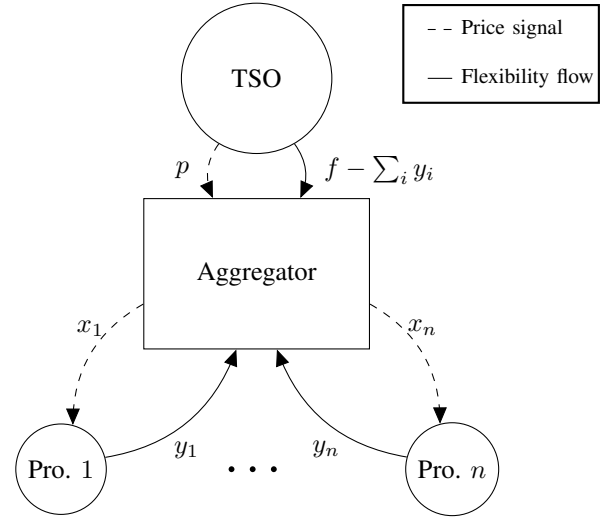}
	\par \textcolor{black}{
	As the agent responsible for supply and demand  balancing in the  \ac{rtbm}, the aggregator has two options to accomplish its goal, namely, to incentivize the prosumers for flexibility provision with the associated cost of $x_iy_i=x_i(y_{i1}+y_{i2})$ or  to buy flexibility from the \ac{tso} with the price $p > 0$.  The aggregator's problem is to find the best strategy given these two options. }
	\par \textcolor{black}{
	Considering the above model, bounds on the proposed price $x_i$ and also the prosumers' optimality conditions,
		we obtain the bilevel optimization problem \eqref{prob0} which has the problem \eqref{prob00} as a constraint for each prosumer:
		\begin{subequations} \label{prob0}
			\begin{alignat}{2} 
				&  \underset{x,y}{\mathrm{min}} \quad  &&\ \phi(x,y)=\sum_{i \in N} x_iy_i+p(f-\sum_{i \in N}y_i) \\
				& \mathrm{subject \ to}\quad  && \ubar{\rho} \le x_i \le \bar{\rho}, \qquad \forall i\in N,\\ \label{0in}
				&&&y_i=y_{i1}+y_{i2} , \qquad \forall i\in N, \\
				&&& \sum_{i \in N}y_i \le f,\\
				&&&  \begin{cases} \label{prob01}
				\underset{y_{i1},y_{i2}}{\mathrm{max}} \quad & x_i(y_{i1}+y_{i2})-(f_i(y_{i1},y_{i2})+b_{i1}y_{i1}+b_{i2}y_{i2})\\ 
				\mathrm{subject \ to} \quad & 0 \le y_{i1} \le m_{i1},\\
				&0 \le y_{i2} \le m_{i2},
				\end{cases}\quad  \forall i\in N,
			\end{alignat}
		\end{subequations}
		where $x$ and $y$ are vectors with components $x_i$ and $y_i$, respectively. 
		Also, $f > 0$ denotes the mismatch between supply and demand in both up- and down-regulation. 
		If the flexibility provided by the prosumers is  $\sum_{i \in N}y_i$ then, the aggregator needs to trade $(f-\sum_{i \in N}y_i)$  with the \ac{tso}.   
		Figure~\ref{fig:interactions} shows these interactions.
		To guarantee a minimum profit for each prosumer and to prevent a high aggregator's payoff, we impose the nonnegative lower and upper bounds $\ubar{\rho}$ and $\bar{\rho}$ on the aggregator's proposed price $x_i$.
		We consider an ex-ante pricing scheme, that is, the \ac{tso} informs the aggregator about the price $p$ prior to the start of each 5-minute interval. }
	\par
	These types of bilevel problems and markets have a strong connection with Stackelberg games \cite{von2010market}, where a leader announces a policy to its followers and then the followers, who are unaware of the outside world,  react by their best response strategy. In other words, the leader has the advantage of anticipating the followers reactions. 
	{A full investigation of such a market in a game-theoretic framework can be found in \cite{tushar2012economics}.}
	\par
	In the setup we consider in this paper, the aggregator's goal is to satisfy its internal imbalance in real-time. However, in other possible settings beyond the scope of this paper, helping the \ac{tso} to satisfy the total system imbalance can also be a goal for the aggregator. Therefore, in that setting the problem formulation for Case~\ref{cs:1} and Case~\ref{cs:3} is given by \eqref{prob0} without considering  \eqref{0in}. In this situation, if $\sum_{i \in N}y_i -f\le 0$, the aggregator pays $p(f-\sum_{i\in N}y_i)$ to the \ac{tso} and if $\sum_{i\in N}y_i - f >0$, then the aggregator receives $p(f-\sum_{i \in N}y_i)$ from the \ac{tso} for providing  flexibility.  
	\subsection{The bilevel market optimization problem with personalized prices and its solution}
	The model above for the aggregator and the prosumers interactions is very close to the bilevel  electricity market models in \cite{zugno2013bilevel,yang2018model,yang2017framework}, where different market technicalities have been considered. Furthermore, we restrict our model to a static case. Despite these differences, our model captures the basic properties of a bilevel market.
	\par
	The aforementioned studies have used two pricing schemes, i.e., the uniform pricing scheme and the personalized pricing scheme interchangeably. However, none of these studies has investigated the optimal solution of the optimization problems with these two pricing scheme in a rigorous mathematical way. In the following two sections, we first show that under the personalized pricing scheme the optimal solution of the bilevel optimization problem can be found by solving an equivalent convex optimization problem. Then, we elaborate on the optimal solution of the bilevel problem with the personalized pricing in contrast to the optimal solution of the  same problem with uniform pricing scheme.
	\section{On the solution of the bilevel electricity market problem with the personalized pricing scheme} \label{sec:ce}
	In general, bilevel optimization problems are very difficult to solve. They have been extensively studied in the framework of \ac{mpec}. We refer to \cite{luo1996mathematical} for a full investigation of \ac{mpec}s.  The  simplest case of a bilevel optimization problem is when both the upper and lower level problems are linear. Even in this simplest case, \cite{hansen1992new} has shown that the problem is strongly NP-hard. Some classes of bilevel optimization problems  can be reformulated  as \ac{mip} problems and solved by commercial software packages \cite{fortuny1981representation}. This approach has been extensively used to solve electricity market optimization problems as a state-of-the-art approach \cite{li2018participation}, \cite{wang2017strategic}.
	\par
	An aggregator can have up to several thousands of prosumers under its contract. To implement an \ac{rtbm} with  5-minute time intervals, the optimal solution of the problem \eqref{prob0} should be found as fast as possible. The increase in the number of the optimization variables, as a result of the growth in the number of the prosumers,  leads to an unacceptable computation time in real-time applications for combinatorial  optimization problems such as \ac{mip} problems.  
	\par
	{In this section, we elaborate on a convex equivalent of the the problem \eqref{prob0}. It should be emphasized that we are not seeking for an algorithm to solve the problem \eqref{prob0}. The contribution here is to introduce a convex reformulation for the bilevel problem \eqref{prob0}. Having a convex equivalent enables us to solve the problem using any algorithm  available in the commercial software packages and find the global optimal solution.
		\textcolor{black}{In what follows, we first show that the bilevel optimization problem \eqref{prob0} is equivalent to a single level optimization problem.
			Then, we prove that under sufficient conditions only one of the \ac{ads} devices of each prosumer becomes active in the \ac{rtbm}. Consequently,we consider the problem of one device per prosumer  and show that the solution of the new problem can be found using a convex equivalent problem.}
		\textcolor{black}{
		\subsection{From bilevel to single-level}
		Given $x_i$ the optimization problem \eqref{prob01}  is a convex optimization problem. Therefore, one can rewrite  \eqref{prob01} as its necessary and sufficient KKT conditions
			\begin{equation}\label{KKT00}
				\begin{aligned}
					&a_{i1}y_{i1} -\sqrt{a_{i1}a_{i2}}y_{i2}+b_{i1}-x_i-\mu_{i1}+\nu_{i1}=0,\\ 
					&-\sqrt{a_{i1}a_{i2}}y_{i1} +a_{i2}y_{i2}+b_{i2}-x_i-\mu_{i2}+\nu_{i2}=0,\\ 
					&0 \le y_{i1}\perp \mu_{i1} \ge 0, \quad 0 \le m_{i1}-y_{i1} \perp \nu_{i1} \ge 0, \\ 
					&0 \le y_{i2}\perp \mu_{i2} \ge 0, \quad 0 \le m_{i2}-y_{i2} \perp \nu_{i2} \ge 0. 
				\end{aligned}
			\end{equation}
			Here $\mu_{i1}$ and $\nu_{i1}$ are the dual variables for the lower bound and upper bound on $y_{i1}$, respectively. Likewise, $\mu_{i2}$ and $\nu_{i2}$ are the dual variables for the lower bound and upper bound on $y_{i2}$, respectively.  Having  \eqref{KKT00}, let us rewrite the bilevel optimization problem \eqref{prob0} as the following single-level optimization problem: 
			\begin{subequations} \label{prob1new}
				\begin{alignat}{2} 
					&  \underset{x,y,\mu,\nu}{\mathrm{min}} \quad  &&\ \phi(x,y)=\sum_{i \in N} x_iy_i+p(f-\sum_{i \in N}y_i) \\
					& \mathrm{subject \ to}\quad  && \ubar{\rho} \le x_i \le \bar{\rho}, \qquad \forall i\in N,\\ 
					&&&y_i=y_{i1}+y_{i2} , \qquad \forall i\in N, \\
					&&& \sum_{i \in N}y_i \le f,\\ \label{prob1newkkt}
					&&&  \begin{cases} 
						&a_{i1}y_{i1} -\sqrt{a_{i1}a_{i2}}y_{i2}+b_{i1}-x_i-\mu_{i1}+\nu_{i1}=0,\\ 
						&-\sqrt{a_{i1}a_{i2}}y_{i1} +a_{i2}y_{i2}+b_{i2}-x_i-\mu_{i2}+\nu_{i2}=0,\\ 
						&0 \le y_{i1}\perp \mu_{i1} \ge 0, \quad 0 \le m_{i1}-y_{i1} \perp \nu_{i1} \ge 0, \\ 
						&0 \le y_{i2}\perp \mu_{i2} \ge 0, \quad 0 \le m_{i2}-y_{i2} \perp \nu_{i2} \ge 0,
					\end{cases}\quad  \forall i\in N.
				\end{alignat}
			\end{subequations}
		Since the KKT conditions are necessary and sufficient for \eqref{prob01}, the next results immediately follows.
			\begin{lemma} \label{lem:evident}
			The optimization problem \eqref{prob0} and \eqref{prob1new} are equivalent.
		\end{lemma}
		\begin{proof}
			See \ref{app}.
		\end{proof}
		\noindent
		In the next subsections, we focus on the optimization problem \eqref{prob1new} as the equivalence of \eqref{prob0}.
		}
		\textcolor{black}{
			\subsection{\ac{ads} device activation}
			In the previous section, we have  built a model based on the fact that  each prosumer can have both \ac{hp} and \ac{mchp}.  However, modeling both types of \ac{ads} devices might not always be necessary as formalized in the following lemma.
			\begin{lemma}\label{lem:act}
				Consider the optimization problem \eqref{prob1new}. Suppose that $\frac{\sqrt{a_{i2}}b_{i1}+\sqrt{a_{i1}}b_{i2}}{\sqrt{a_{i1}}+\sqrt{a_{i2}}}> \bar{\rho}$.  Then, the  following  statements hold.
				\begin{enumerate}[label=\Roman*)]
					\item $y_{i1}^*y_{i2}^*=0$. \label{lem:act1}
					\item If $b_{i1} \le 0 $ and $ b_{i2} \ge 0$, $y_{i2}^*=0$. That is the $i$th prosumer's mCHP does not provide flexibility in down-regulation.  
					\item If $b_{i1} \ge0 $ and $ b_{i2} \le 0$, $y_{i1}^*=0$. That is the $i$th prosumer's HP does not provide flexibility in up-regulation.  
				\end{enumerate}
			\end{lemma}
		\begin{proof}
			See \ref{app}.
		\end{proof}		
		Motivated by the lemma above, hereafter, we assume that each prosumer has either an \ac{hp} or \ac{mchp}.  Therefore, \eqref{prob1newkkt} can be rewritten as
		\begin{equation}\label{KKT00new} 
			\begin{aligned} 
				a_iy_i+b_i-x_i&-\mu_i+\nu_i=0,\\
				0 \le y_i &\perp \mu_i \ge 0, \\
				0 \le m_i-y_i &\perp \nu_i \ge 0.
			\end{aligned}
		\end{equation}
	Note that to  ease  the notation, we have  dropped $1$ and $2$ in the subscripts related to each prosumer since it only has one \ac{ads} device.
		}
		Solving the parametric linear complementarity
			problem \eqref{KKT00new} analytically leads to the following piece-wise linear map from $x_i$ to $(y_i,\mu_i,\nu_i)$:
			\begin{equation} \label{sol01}
				(y_i,\mu_i,\nu_i)=\begin{dcases}
					(0,b_i-x_i,0) \quad &  x_i < b_i,  \\
					(\frac{x_i-b_i}{a_i},0,0) \quad & b_i \le x_i \le a_im_i+b_i,  \\
					(m_i,0,x_i-a_im_i-b_i)  \quad & x_i > a_im_i+b_i.
				\end{dcases} 
		\end{equation} 
		This allows us to   rewrite the optimization problem \eqref{prob1new} as the following piece-wise quadratic optimization problem:
		\begin{subequations} \label{prob1}
			\begin{alignat}{2}
				&  \underset{x,y,\mu,\nu}{\mathrm{min}} \quad  && \phi(x,y) = \sum_{i \in N} x_iy_i+p(f-\sum_{i \in N}y_i)\label{prob1a} \\ \label{prob1b}
				& \mathrm{subject \ to}\quad  && \ubar{\rho} \le x_i \le \bar{\rho}, \qquad \forall i\in N\\ \label{prob1c}
				&&& \sum_{i \in N}y_i \le f,\\ \label{prob1d}
				&&& \mkern-90mu (y_i,\mu_i,\nu_i)=\begin{dcases}
					(0,b_i-x_i,0) \quad &  x_i < b_i,  \\
					(\frac{x_i-b_i}{a_i},0,0) \quad & b_i \le x_i \le a_im_i+b_i,  \\
					(m_i,0,x_i-a_im_i-b_i)  \quad & x_i > a_im_i+b_i,
				\end{dcases}  \quad \forall i \in N,
			\end{alignat}
		\end{subequations}
		{\subsection{On the convexity of single-level optimization problem}}
		Here, we elaborate on the solution of the optimization problem \eqref{prob1}. It turns out  under some specific conditions, the optimization problem \eqref{prob1} has  trivial optimal solution for some $i\in N$. The following lemma investigates these specific conditions.
		\begin{lemma}\label{lem:trivialsol}
			Consider the optimization problem \eqref{prob1}. Then, the following statements hold.
			\begin{enumerate}[label=\Roman*)]
				\item Suppose $b_i > \bar{\rho}$ for some $i\in N$. Then,   $x_i^* \in [ \ubar{\rho},  \bar{\rho}]$, $y_i^*=0$, $\mu_i^*=b_i-x_i^*$ and $\nu_i^*=0$.
				\item Suppose $ \ubar{\rho}> a_im_i+b_i$ for some $i\in N$. Then, $x_i^* = \ubar{\rho}$, $y_i^*=m_i$, $\mu_i^*=0$ and $\nu_i^*=x_i^*-a_im_i-b_i$.
			\end{enumerate}
		\end{lemma}
		\begin{proof}
			See \ref{app}.
		\end{proof}
		The above lemma  shows that if $b_i > \bar{\rho}$ or $ \ubar{\rho}> a_im_i+b_i$ for some $i \in N$, we can find the optimal solutions without solving any optimization problem. Then, the following question arises immediately: What if none of the conditions in Lemma~\ref{lem:trivialsol} holds?
		This question is answered by the following example and the results after that.
		\begin{example}
			Suppose a two-dimensional case of the problem \eqref{prob1}  where $a_1=a_2=1$, $b_1=b_2=2$, $m_1=m_2=6$, $\ubar{\rho}=0$, $\bar{\rho}=10$, $p=10$, $f=30$. It is obvious, based on  Lemma~\ref{lem:trivialsol} and the parameters,  that  this problem has no trivial solutions. Figure \ref{nonconvex} depicts  objective function of the problem \eqref{prob1}  with these parameters. As can be seen, the objective function is non-convex and consists of several convex quadratic functions.
			Note that its minimum coincides with the minimum of the convex quadratic problem obtained from \eqref{prob1} by taking $y_i=\frac{x_i-b_i}{a_i}$ and $\mu_i=\nu_i=0$ with $b_i\le x_i \le a_im_i+b_i$  for $i\in\{1,2\}$.
		\end{example}
		\begin{figure*}[t]
			\centering
			\subfloat[The   non-convex objective function.]{%
				\includegraphics[width=0.45\textwidth]{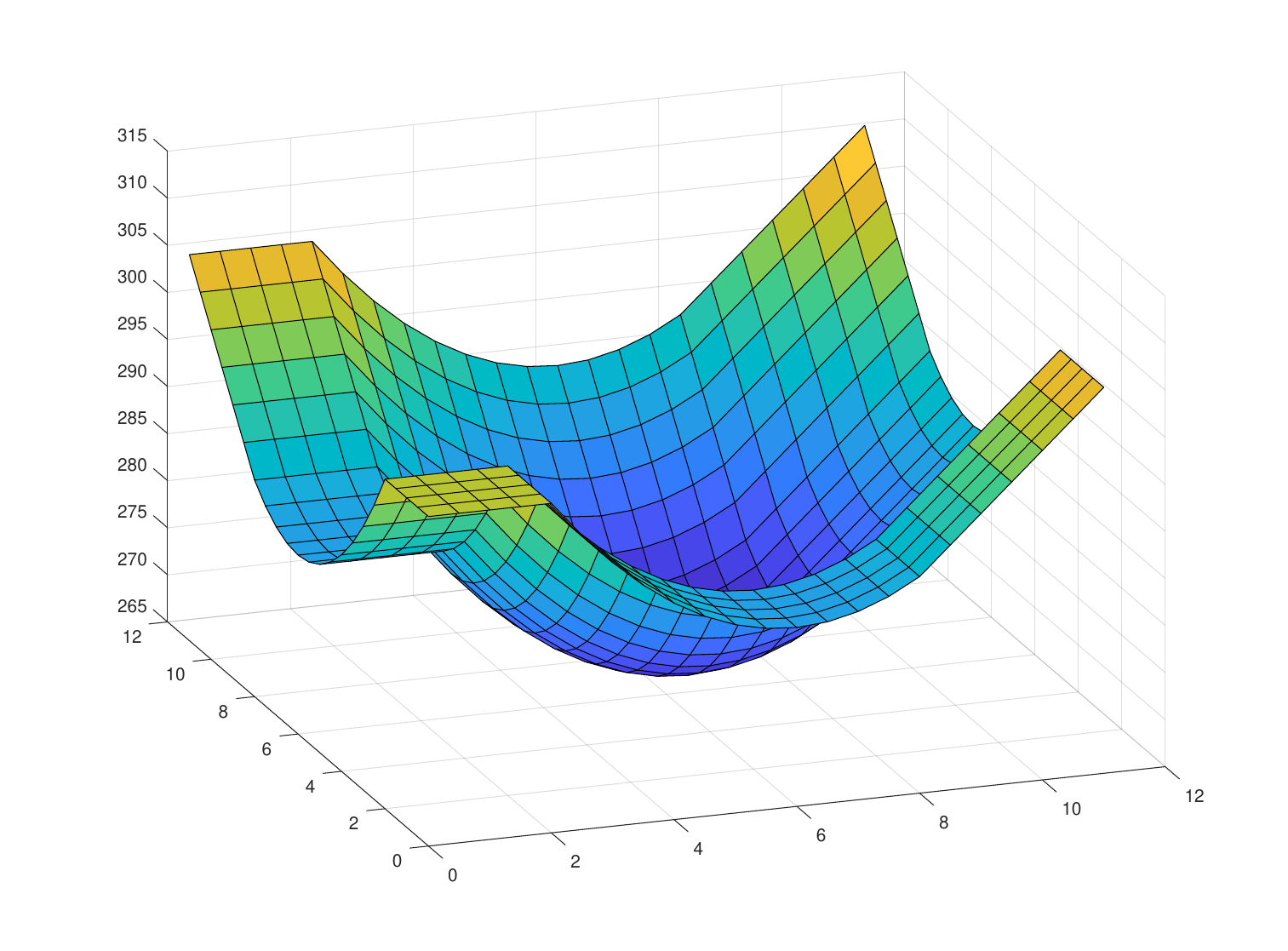}\label{nonconvex}.	} 
			\subfloat[The optimal  point (depicted by red circle) in  the restricted feasibility region.]{%
				\includegraphics[width=0.45\textwidth]{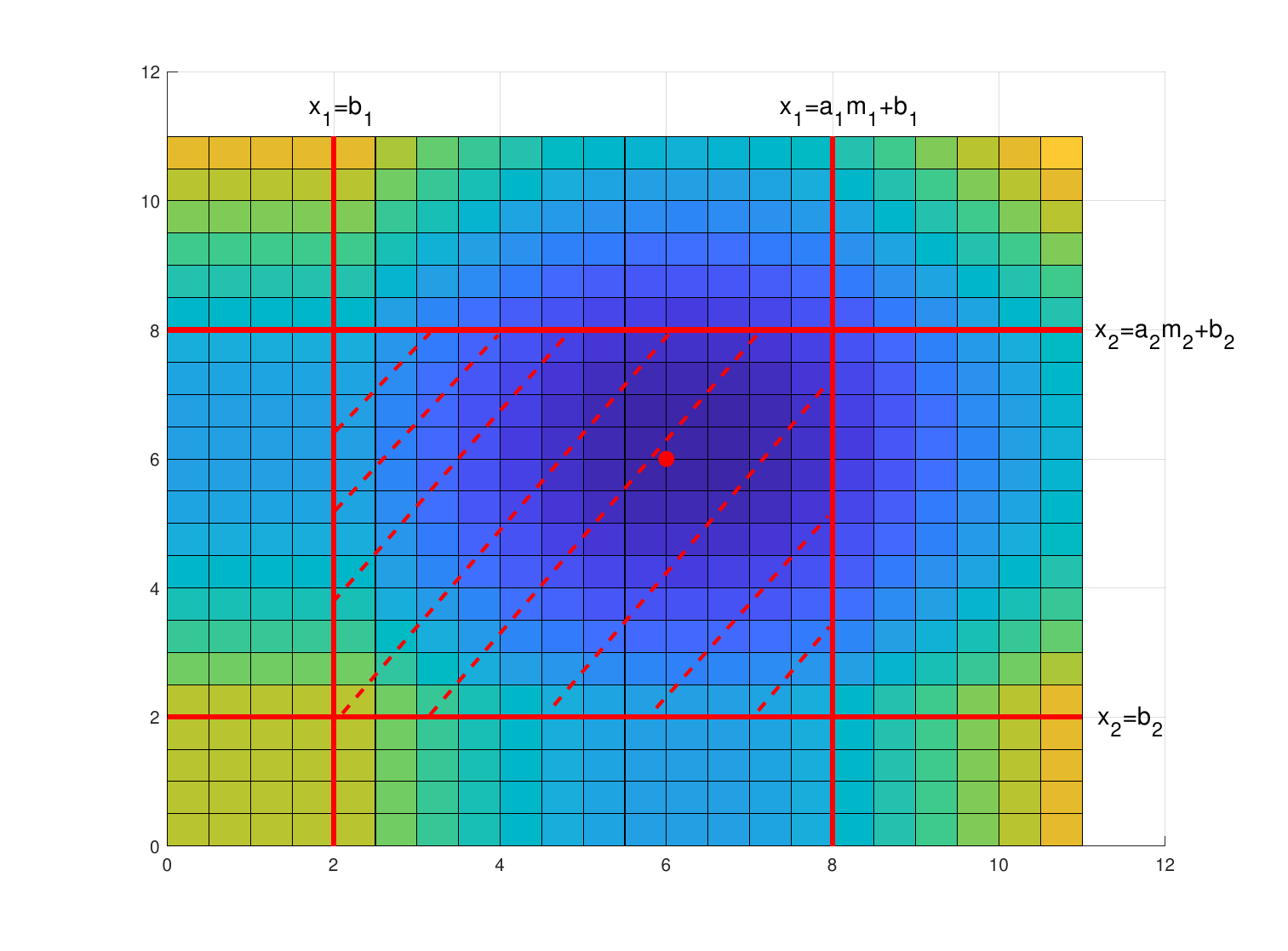}\label{nonconvex2}.	}
			\caption{Two-dimensional case example.}
		\end{figure*}  
		Motivated by this example, we consider the following convex quadratic problem by taking $\mu_i=\nu_i=0$ for all $i \in N$:
		\begin{subequations} \label{prob2}
			\begin{alignat}{2}
				&  \underset{x,y}{\mathrm{min}} \quad  && \phi(x,y) = \sum_{i \in N} {x}_i{y}_i+p(f-\sum_{i \in N}{y}_i) \\ \label{prob2b}
				& \mathrm{subject \ to}\quad  && \ubar{\rho} \le {x}_i\le \bar{\rho}, \qquad \forall i \in N\\ \label{prob2c}
				&&&\sum_{i \in N}{y}_i \le f,\\ \label{prob2d}
				&&&  {y}_i=\frac{{x}_i-b_i}{a_i}, \quad  b_i \le {x}_i \le a_im_i+b_i, \quad \forall i \in N. 
			\end{alignat}
		\end{subequations}
		It appears that a global minimum of the nonconvex problem \eqref{prob1}  can be found by solving the convex problem \eqref{prob2}.
		\begin{lemma} \label{lem:convex}
			Assume $\ubar{\rho} -a_im_i \le b_i \le \bar{\rho}$ for all $i\in N$. Then, there exists an optimal solution $x^*, y^*, \mu^*$ and $\nu^*$ for \eqref{prob1}  such that $\mu^*=\nu^*=0$ and the same $x^*$ and $y^*$ are also the minimizers of the convex quadratic problem \eqref{prob2}.
		\end{lemma}
		\begin{proof}
			See \ref{app}.
		\end{proof}
		\begin{remark}
			The piece-wise linear constraint \eqref{prob1d}  makes the problem \eqref{prob1} a piece-wise quadratic optimization problem with $3^n$ quadratic problems where $n$ is the number of prosumers. Lemma  \ref{lem:convex} proves that under the assumption $\ubar{\rho} -a_im_i \le b_i \le \bar{\rho}$ for all $i\in N$, one of these $3^n$ problems  always attains the global optimum. 
		\end{remark}
		Now, we are in a position to state the main results of this paper.
		\begin{theorem}\label{main}
			Consider the optimization problem \eqref{prob1}. Let $\alpha=\{i \in N \mid b_i > \bar{\rho}\}$, $\beta=\{i \in N \mid a_im_i+b_i < \ubar{\rho}\}$ and $\theta=\{i \in N \mid \ubar{\rho} -a_im_i \le b_i \le \bar{\rho}\}$. Then,  
			\begin{gather}
				x_i^* \in [\ubar{\rho}, \bar{\rho}], \ y_i^*=0,  \quad \forall \ i \in \alpha, \\
				x_i^*= \ubar{\rho} , \ y_i^*=m_i, \quad \forall \ i  \in \beta, 
			\end{gather}
			and $x_i^*, y_i^*$ for all $i \in \theta$ are the minimizers of the following convex problem:
			\begin{subequations} \label{conopt}
				\begin{alignat}{2}
					& \min_{\substack{x_i, y_i\\ \forall \ i \in \theta}} \quad  && \phi(x,y) = \sum_{i \in \theta} {x}_i{y}_i+\sum_{i \in \beta} \ubar{\rho}m_i+p(f-\sum_{i \in \theta}{y}_i-\sum_{i \in \beta}{m}_i) \\ 
					& \mathrm{subject \ to}\quad  && \ubar{\rho} \le {x}_i\le \bar{\rho}, \qquad \forall i \in \theta,\\ 
					&&&\sum_{i \in \theta}{y}_i \le f - \sum_{i \in \beta}{m}_i ,\\ 
					&&&  {y}_i=\frac{{x}_i-b_i}{a_i}, \quad  b_i \le {x}_i \le a_im_i+b_i, \quad \forall i \in \theta.
				\end{alignat}
			\end{subequations}
		\end{theorem}	
		\begin{proof}
			The proof for the optimal solutions of the subsets $\alpha$ and $\beta$ immediately follows from Lemma~\ref{lem:trivialsol}. Eliminating this trivial solutions, the proof for the minimizers of indices in $\theta$ follows from Lemma~\ref{lem:convex}.
		\end{proof}
		Another advantage of using the convex optimization problem \eqref{conopt} over the bilevel one stems from privacy considerations. Indeed, the aggregator needs to have all information about the prosumers to the bilevel problem in a centralized way. However, the prosumers may not be willing to share their information with third parties due to privacy concerns. Since Theorem~\ref{main} allows a distibuted solution to find the optimum (see \cite{Bertsekas/99}), such privacy concerns are not an obstacle for solving the problem \eqref{prob2} or \eqref{conopt}.
		\section{Personalized pricing vs. Uniform pricing}\label{sec:vs}
		In the setup we have considered so far in this work, a personalized pricing scheme is implemented. This means that the aggregator proposes different prices to each prosumer to minimize its cost. However, in another scenario, one can consider a uniform pricing  scheme where the aggregator proposes the same price to all the prosumers \cite{zugno2013bilevel}. These two pricing schemes are very well-known in microeconomics literrature \cite{phillips2021pricing}. In what follows, we investigate the advantages of personalized pricing over uniform pricing  in the defined balancing market. For this purpose, we first (re)write the problems for these schemes. The optimization problem PP corresponds to the personalized pricing  scheme:
		\begin{subequations} \label{probpp}
			\begin{alignat}{2} 
				\mathrm{PP:}\qquad&  \underset{x,y,\mu,\nu}{\mathrm{min}} \quad  &&\ \phi(x,y)=\sum_{i\in N} (x_i-p) y_i+pf \\
				& \mathrm{subject \ to}\quad  && 0 \le x \le \bar{\rho} \qquad \forall i\in N, \label{probpp1}\\ 
				&&& \sum_{i\in N}y_i \le f,\\
				&&& y_i=\frac{x_i-b_i+\mu_i-\nu_i}{a_i},\qquad \forall i\in N,\\
				&&&	0 \le y_i \perp \mu_i \ge 0,\qquad \forall i\in N, \\
				&&& 0 \le m_i-y_i \perp \nu_i \ge 0,\qquad \forall i\in N.
			\end{alignat}
		\end{subequations}
		Note that this is a reformulation of the problem \eqref{prob1}. For simplicity, we consider the parameter $\ubar{\rho}$ equal to zero, although all the following analyses can be verified for arbitrary $\ubar{\rho}$. 
		Similar to the problem above, we define the problem UP for the uniform pricing scheme. Here all the proposed prices to the prosumers are equal and it is denoted by the scalar decision variable $x$:
		\begin{subequations} \label{probup}
			\begin{alignat}{2} 
				\mathrm{UP:}\qquad&  \underset{x,y,\mu,\nu}{\mathrm{min}} \quad  &&\ \phi(x,y)=\sum_{i \in N}(x-p) y_i+pf \\
				& \mathrm{subject \ to}\quad  &&  0 \le x \le \bar{\rho}, \label{probup1}\\ 
				&&& \sum_{i\in N}y_i \le f,\\
				&&& y_i=\frac{x-b_i+\mu_i-\nu_i}{a_i},\qquad \forall i\in N,\\
				&&&	0 \le y_i \perp \mu_i \ge 0,\qquad \forall i\in N, \\
				&&& 0 \le m_i-y_i \perp \nu_i \ge 0,\qquad \forall i\in N.
			\end{alignat}
		\end{subequations}
		\par
		One of the main benefits of the personalized pricing scheme is that it leads to a lower or equal balancing cost. The next proposition states this advantage.
		\begin{proposition}\label{pro:cost}
			The aggregator's optimal cost in the personalized pricing scheme is less than or equal than its optimal cost in the uniform pricing scheme, i.e.,
			\[\phi_{\mathrm{PP}}^* \le \phi_{\mathrm{UP}}^*.\] 
		\end{proposition}
		\begin{proof}
			One can rewrite the problem \eqref{probup} by replacing $x$ by $x_i$ and add an extra constraint as
			\[x_1=x_2=\dots=x_n.\]
			Therefore, the feasible set of the problem UP is a subset of the feasible set of the problem PP. This concludes that $\phi_{\mathrm{PP}}^* \le \phi_{\mathrm{UP}}^*$.
		\end{proof}
		Having a less balancing cost for the aggregator is not the only superior aspect of the personalized pricing scheme. 
		The next proposition shows that under this pricing scheme more prosumers contribute to the balancing market.
		\begin{proposition} \label{pro:num}
			Let $n_{\mathrm{PP}}(N)$ and $n_{\mathrm{UP}}(N)$ be the number of prosumers who participate in the personalized  and uniform pricing scheme, respectively. Then, $n_{\mathrm{UP}}(N) \le n_{\mathrm{PP}}(N)  $.
		\end{proposition}
		To prove the proposition above, we need some auxiliary results. The following lemmas concerning the optimization problems PP and UP play an essential role in the proof of Proposition~\ref{pro:num}. 
		\begin{lemma}\label{lem:bneg} 
			Consider the optimization problems PP and UP. 
			Then the following two statements hold.
			\begin{enumerate}[label=\Roman*)]
				\item Let $b_i <0$ for some $i \in N$. Then, the optimal solution $y_i^*$ is positive for both problems.
				\item Let $b_i > \bar{\rho}$ for some $i \in N$. Then,  the optimal solution $y_i^*$ is zero for both problems.
			\end{enumerate}	
		\end{lemma}
		\begin{proof}
			See \ref{app}.
		\end{proof}
		\begin{lemma}\label{lem:pp}
			Consider the optimization problem PP. Suppose that $0 \le b_i \le  \bar{\rho}$ for all $i \in N$. If $p >b_i$, then $y_i^* > 0$.
		\end{lemma}
		\begin{proof}
			See \ref{app}.
		\end{proof}
		Lemma~\ref{lem:pp} provides a \textit{sufficient} condition for contribution of each prosumer in the personalized pricing scheme,
		whereas the next one provides a \textit{necessary} condition for contribution of each prosumer in the uniform pricing scheme.
		\begin{lemma}\label{lem:up}
			Consider the optimization problem UP. Suppose that $0 \le b_i \le  \bar{\rho}$ for all $i\in N$. Also, suppose the sets  $\gamma=\{i\in N \mid y_i^*>0\}$ and $\bar{\gamma}=\{i\in N \mid  y_i^*=0\}$ are given. Then, 
			$p > b_i$ for all  $i \in \gamma$.
		\end{lemma}
		\begin{proof}
			See \ref{app}.
		\end{proof}
		\textcolor{black}{
			\begin{remark} \label{remarkgammah}
				Note that in Lemma~\ref{lem:up}, $(p-b_i)$ is sign-indefinite for $i \in \bar{\gamma}$. Therefore, we can argue that there exists $\hat \gamma$ such that $N \supseteq \hat \gamma \supseteq \gamma$ and $p>b_i$ for all $i\in \hat{\gamma}$.
		\end{remark}}
		Now, we are in a position to prove Proposition~\ref{pro:num}.
		\textcolor{black}{
			\begin{proof}[Proof of Proposition~\ref{pro:num}]
				Define the set, $\alpha=\{i \in N \mid b_i <0\}$, $\beta=\{i \in N \mid 0 \le b_i \le  \bar{\rho}\}$ and $\theta=\{i \in N \mid b_i >\bar{\rho}\}$. Due to Lemma~\ref{lem:bneg}, $n_\mathrm{PP}(\alpha)=n_\mathrm{UP}(\alpha)=\lvert \alpha \rvert$ and $n_\mathrm{PP}(\theta)=n_\mathrm{UP}(\theta)=0$. Now, suppose  that $n_\mathrm{UP}(\beta)$ is given. Then, based on Lemma~\ref{lem:up}, $p > b_i$ holds for all $i \in \hat \gamma$ where $\hat \gamma$ is defined in Remark~\ref{remarkgammah}. As a result, due to Lemma~\ref{lem:pp}, $n_\mathrm{PP}(\beta) \ge n_\mathrm{UP}(\beta)$. Consequently, we have $n_\mathrm{PP}(N) \ge n_\mathrm{UP}(N)$. 
		\end{proof}}
		\textcolor{black}{The profit of a single prosumer in the personalized pricing scheme might be higher or lower than its profit in the uniform pricing scheme. Nonetheless,  Proposition~\ref{pro:num} states that the chance of participation of a prosumer and having revenue in the balancing market is higher in the personalized scheme.}
		\section{Simulations} \label{sec:sim}
		In this section, first
		we evaluate the performance of our convex equivalent problem for the \ac{rtbm} in terms of  computation time and  optimality. We use the state-of-the-art \ac{mip}-based approach in \cite{fortuny1981representation} as a benchmark for this evaluation. 
		Next, we compare the aggregator's cost and prosumers' contribution under two  schemes: personalized and uniform pricing.
		\par
		For  simulation purposes, we consider one type of \ac{hp} and two types of \ac{mchp} technologies for the prosumers. We assume that  half of the prosumers have \ac{hp} and the other half are equipped with \ac{mchp}. We assign to each prosumer a specific technology of \ac{hp} or \ac{mchp}, randomly. Tables \ref{table:hp} and \ref{table:mchp} show the data regarding these types and also their corresponding $\abs{b_i}$ parameters.
		The supplier gas and electricity prices are based on data from \cite{retailerprice} 
		for the  Netherlands and  equal to $0.0861 \ \text{\euro}/\mathrm{kWh}$ and $0.1707 \ \text{\euro}/\mathrm{kWh}$, respectively. The  price $p$ for both up- and down-regulation  is set to $0.7 \ \text{\euro}/\mathrm{kWh}$ based on the settlement price data of TenneT from \cite{settlementprice} 
		for a period where the \ac{tso} is under high stress.   It should be noted that the \ac{tso} informs the aggregators about this price ex-ante. Also, we assume that $\bar{\rho}=p=0.7$ and $\ubar{\rho}=0$.
		\par 
		All optimization problems are implemented in MATLAB r2018b and solved by the Gurobi Optimizer  \cite{gurobi}. 
		The simulations were run on four Intel Xeon 2.6 GHz cores and 1024 GB internal memory of the Peregrine high performance computing cluster of the University of Groningen.   
		\begin{table*}[t] 
			\begin{center}
				\caption{The parameters for different \ac{hp} technologies.}
				\label{table:hp}
				\begin{tabular}{|c|c|c|}
					\hline
					\ac{hp} type & Nominal electricity input power & $\abs{b_i}$  \\ \hline
					$1 $ & $1.1 \ kW$ &  $0.1707 \ \text{\euro}/kWh$
					\\ \hline
				\end{tabular}
			\end{center}
		\end{table*}
		\begin{table*}[t] 
			\begin{center}
				\caption{The parameters for different \ac{mchp} technologies.}
				\label{table:mchp}
				\resizebox{\columnwidth}{!}{%
					\begin{tabular}{|c|c|c|c|}
						\hline
						\ac{mchp} type & Nominal input power & Nominal electricity output power & $\abs{b_i}$  \\ \hline
						$1$ & $8 \ kW$ & $1 \ kW$ & $0.6888 \ \text{\euro}/{kWh}$
						\\ \hline
						$2$ & $4.7 \ kW$ & $0.8 \ kW$ &  $0.5088 \ \text{\euro}/{kWh}$
						\\ \hline
				\end{tabular}}
			\end{center}
		\end{table*}
		\subsection{Computation time and optimality comparison}
		Here, we first define the \ac{mip} formulation of the problem \eqref{prob0}. This formulation is used as a benchmark to evaluate the computational efficiency of the convex equivalent of the bilevel problem. 
		By introducing  dual variables $\lambda_{1i}, \lambda_{2i}$,   auxiliary binary variables $z_i, w_i$   and  a sufficiently large constant $M$, the problem~\eqref{prob0} can be turned to an \ac{mip} problem  as
		\begin{subequations} \nonumber
			\begin{alignat}{3} 
				&  \underset{w,y,z,\lambda_{1},\lambda_{2}}{\mathrm{min}} \quad  && \sum_{i} (a_iy_i^2+(b_i-p)y_i+m_i\lambda_{2i})+pf \\
				& \mathrm{subject \ to}\quad  &&  \sum_{i}y_i \le f, \\
				&&&  \begin{cases}
					x_i=a_iy_i+b_i-\lambda_{1i}+\lambda_{2i} \ge 0, \\
					x_i=a_iy_i+b_i-\lambda_{1i}+\lambda_{2i} \le \bar{\rho}, \\
					0 \le y_i \le Mz_i, \\
					0 \le m_i-y_i \le Mw_i,\\
					0 \le \lambda_{1i} \le M(1-z_i),\\
					0 \le \lambda_{2i} \le M(1-w_i),
				\end{cases}
				\quad  \forall i\in N.
			\end{alignat}
		\end{subequations}
		Details of this approach can be found in \cite{fortuny1981representation}.  
		\par 
		The \ac{mip}  solvers  use complicated heuristic methods to find the optimal solution. Moreover, the computation time for computing  an optimal solution is highly related to specific parameters of the problem.      To find a rough estimate of the optimization run time, we implement a set of 1000 Monte Carlo simulations with uniformly generated random parameters $a_i$, $m_i$ and $f$  for the optimization problem. This is done for different numbers of prosumers. Table \ref{table:results} summarizes the run time results for these Monte Carlo scenarios. The last  column of this table shows the number of scenarios (out of 1000 Monte Carlo scenarios)  where the \ac{mip} problem leads to an infeasible solution or an optimal solution with higher cost than the convex problem.
		\par
		The computation time for the convex optimization problem grows approximately linear with respect to the number of prosumers. This can be seen from the average run time in Table~\ref{table:results} for the convex formulation. If we consider $30000$ as the typical number of prosumers for an aggregator, then the average and the maximum run time are acceptable for a real-time application with $5$-minute time interval.    However, this is not the case for an \ac{mip} formulation. Figure~\ref{fig:runtime} and Table~\ref{table:results} show that the average and maximum computation time of \ac{mip} is not suitable for a real-time market since the computational time grows approximately exponentially.   
		Moreover,  there are some cases that the \ac{mip} formulation with high number of optimization variables does not converge to the global optimal or even to feasible solution.  
		{This is shown on the last column of Table~\ref{table:results}. For instance, for $10$-prosumer case, both the \ac{mip} and convex formulation have the same optimal solution in all $1000$ random scenarios. Nevertheless, in $30000$-prosumer case, the \ac{mip} formulation converges to a higher minimum cost or an infeasible solution with respect to the convex formulation in $40$ out of $1000$ random scenarios of the simulations. It is clear that in the rest $960$ scenarios both the formulations have the same optimal solution.} 
		\begin{table*}[t]
			\begin{center}
				\caption{Simulation run time and optimality comparison.}
				\label{table:results}
				\resizebox{\columnwidth}{!}{
					\begin{tabular}{|c|c|c|c|c|c|}
						\hline
						\multirow{ 2}{*}{Number of Prosumers} & \multicolumn{2}{c|}{Convex formulation run time} & \multicolumn{2}{c|}{\ac{mip} formulation run time} &{Number of scenarios with infeasible} \\ \cline{2-5}
						&Average (sec)    & Maximum (sec)  & Average  (sec)   & Maximum (sec) & or  non-optimal solution for \ac{mip}  \\ \hline
						10&0.0006&0.0010&0.0016&0.0039& 0 \\\hline
						100&0.0012&0.0017&0.0032&0.0058& 0 \\\hline
						1000&0.0038&0.0076&0.0128&0.0371& 1\\\hline
						10000&0.0344&0.0484&0.5548&9.0352& 11\\\hline
						20000&0.0772&0.1231&3.9498&59.1761& 27\\\hline
						30000&0.1161&0.1834&11.1190&161.7937& 40\\\hline
				\end{tabular}}
			\end{center}
		\end{table*}
		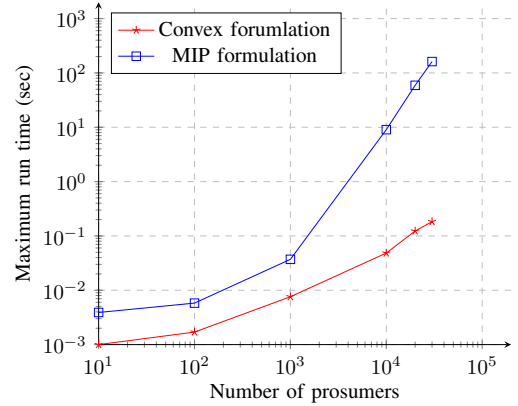
\begin{figure}[t]
			\centering  
			\begin{tikzpicture}[scale=0.8]
				\begin{loglogaxis}[
					axis lines = left,
					xlabel={Number of prosumers},
					ylabel={Maximum run time (sec)},
					xmin=0, xmax=200000,
					ymin=0, ymax=2000,
					xtick={1,10,100,1000,10000,100000},
					ytick={0.001,0.01,0.1,1,10,100,1000},
					ymajorgrids=true,
					xmajorgrids=true,
					grid style=dashed,
					legend pos= north west,
					]
					
					\addplot[
					color=red,
					mark=star,
					]
					coordinates {
						(10,0.0010)(100,0.0017)(1000,0.0076)(10000,0.0484)(20000,0.1231)(30000,0.1834)
					};
					\addplot[
					color=blue,
					mark=square,
					]
					coordinates {
						(10,0.0039)(100,0.0058)(1000,0.0371)(10000,9.0352)(20000,59.1761)(30000,161.7937)
					};
					\legend{Convex formulation,\ac{mip} formulation}
				\end{loglogaxis}
			\end{tikzpicture}
			\caption{Maximum run  time: Convex vs. \ac{mip} formulation.}
			\label{fig:runtime}
		\end{figure}
		\subsection{Pricing schemes comparison}
		\begin{table*}[t] 
			\begin{center}
				\caption{The prosumers' parameters for pricing scheme comparison.}
				\label{table:pro}
				\resizebox{.7\columnwidth}{!}
				{%
					\begin{tabular}{|c|c|c|c|c|}
						\hline
						Pro. number & \ac{mchp} type & $a_i \ (\text{\euro}/kWh^2)$ & $b_i \ (\text{\euro}/{kWh}) $  & $m_i  \ (kWh) $ \\ \hline
						$1$ & Type $1$ & $2$  & $0.6888  $& $ 0.08  $
						\\ \hline
						$2$ & Type $1$ & $5$ & $0.6888 $& $ 0.05$
						\\ \hline
						$3$ & Type $2$ & $10$ & $0.5088 $& $ 0.02$
						\\ \hline
						$4$ & Type $2$ & $5 $ & $0.5088  $& $ 0.01 $
						\\ \hline
						$5$ & Type $2$ & $20 $ & $0.5088  $& $ 0.025  $
						\\ \hline
				\end{tabular}}
			\end{center}
		\end{table*}
		This subsection is devoted to show the validity of   Proposition~\ref{pro:cost}~and~\ref{pro:num}.
		We consider a case where the aggregator and TSO are in down regulation. The total number of prosumers is assumed to be $5$ and all are equipped with  \ac{mchp}s.
		The full details  of prosumers' parameters are presented in Table~\ref{table:pro}.
		The requested flexibility $f$ is $0.05 \ \mathrm{kWh}$.  The results for both pricing schemes are demonstrated in Table~\ref{table:pricingresult}.
		\par
		The optimal results in Table~\ref{table:pricingresult} shows that all prosumers  contribute  to the balancing market under the personalized pricing scheme. However, in the uniform pricing scheme, only the prosumer number $3$, $4$ and $5$ provide flexibility. Furthermore, the aggregator's optimal cost in the personalized pricing scheme is less than its optimal cost in the uniform pricing scheme. Indeed, this is inline with what is claimed in Section~\ref{sec:vs}.
		\begin{table*}[t]
			\begin{center}
				\caption{Optimal price and flexibility: Personalized pricing vs. Uniform pricing}
				\label{table:pricingresult}
				\resizebox{\columnwidth}{!}{
					\begin{tabular}{|c|c|c|c|c|}
						\hline
						\multirow{ 2}{*}{Prosumer number} & \multicolumn{2}{c|}{Personalized pricing scheme} & \multicolumn{2}{c|}{Uniform pricing scheme}  \\ \cline{2-5}
						&$x_i^* \ (\text{\euro}/kWh)$ &$y_i^* \ (kWh)$&$x_i^* \ (\text{\euro}/kWh)$ & $y_i^* \ (kWh)$  \\ \hline
						$1$&$0.6944$&$0.0028$&$0.5711$&$0$\\ \hline
						$2$&$0.6944$&$0.0011$&$0.5711$&$0$ \\ \hline
						$3$&$0.6044$&$0.0096$&$0.5711$&$0.0062$\\ \hline
						$4$&$0.5588$&$0.0010$&$0.5711$&$0.0100$\\ \hline
						$5$&$0.6044$&$0.0048$&$0.5711$&$0.0031$ \\ \hline
						Agg. cost ({\euro}) & \multicolumn{2}{c|}{$0.0322$} & \multicolumn{2}{c|}{$0.0335$}  \\ \hline
				\end{tabular}}
			\end{center}
		\end{table*}
		\section{Conclusions}\label{sec:con}
		In this paper, we have developed  a market with   a \ac{tso},  an aggregator and prosumers to address real-time balancing. We have modeled the corresponding economic optimization problem of a self-interested aggregator and prosumers as a bilevel optimization problem under a personalized pricing  scheme. Generally,  bilevel optimization problems are non-convex. We have shown that it suffices to solve a specific convex optimization problem to find the global optimum of the original bilevel optimization problem. In contrast to existing approaches  (e.g., \ac{mip}), the convex equivalent of the bilevel optimization problem has very low computation time and is therefore preferable in real-time.
		Low computation time and global optimality  are not the only advantages of having a convex equivalent for the bilevel optimization.   Centralized  aggregator control over the whole community of prosumers can be a difficult task, especially when the number of prosumers is very high. However, having a convex formulation for the balancing problem opens up new horizons in  decentralized and distributed control and optimization.
		\par
		Also, we have compared the optimal solutions for two pricing scheme, i.e., personalized and uniform pricing scheme.  We have shown, in a rigorous mathematical way, that under the personalized pricing scheme more prosumers contribute to the balancing market and the aggregator's optimal cost is less.
		\section*{Acknowledgment}
		This research was funded by the NWO (The Netherlands Organisation for Scientific Research) Energy System Integration project "Hierarchical and distributed optimal control of integrated energy systems" [647.002.002].

		\bibliography{mybib}
		\appendix
		\section{Proofs of  lemmas} \label{app}
		\begin{proof}[Proof of Lemma~\ref{lem:evident}]
			The proof is evident from the fact that the optimization problem \eqref{prob01} is a convex optimization problem in $y_{i1}$ and $y_{i2}$ for any given $x_i$ and the KKT conditions are necessary and sufficient for this problem.
		\end{proof}
		}{
		\begin{proof}[Proof of Lemma~\ref{lem:act}]
			Let $x_i^*, y_{i1}^*, y_{i2}^*,  \mu_{i1}^*, \mu_{i2}^*,  \nu_{i1}^*$ and $\nu_{i2}^*$ be the optimal solution of the problem \eqref{prob1new}.
			\begin{enumerate}[label=\Roman*:]
				\item  Suppose, on the contrary, that $y_{i1}^*y_{i2}^*\neq0$. Therefore, $\mu_{i1}^*= \mu_{i2}^*=0$ and we have 
				\begin{gather} \label{act1}
					a_{i1}y_{i1}^* -\sqrt{a_{i1}a_{i2}}y^*_{i2}+b_{i1}-x_i^*+\nu_{i1}^*=0,\\  \label{act2}
					-\sqrt{a_{i1}a_{i2}}y_{i1}^* +a_{i2}y_{i2}^*+b_{i2}-x_i^*+\nu_{i2}^*=0.
				\end{gather}
				We multiply \eqref{act1} by $\sqrt{a_{i2}}$ and \eqref{act2} by $\sqrt{a_{i1}}$. By adding these two terms, we get 
				\begin{equation*}
					x_i^* \ge \frac{\sqrt{a_{i2}}b_{i1}+\sqrt{a_{i1}}b_{i2}}{\sqrt{a_{i1}}+\sqrt{a_{i2}}},
				\end{equation*}
				which is a contradiction since $x_i^* \le \bar\rho$. Therefore, either $y_{i1}^*$ or $y_{i2}^*$ is zero.
				\item Since $b_{i2}\ge b_{i1}$, we have $b_{i2} \ge \frac{\sqrt{a_{i2}}b_{i1}+\sqrt{a_{i1}}b_{i2}}{\sqrt{a_{i1}}+\sqrt{a_{i2}}} > \bar{\rho}$. Suppose, on the contrary, that $y_{i2}^*>0$. Then, based on  item~I, $y_{i1}^*=0$. 
				This point should satisfy the constraints of \eqref{prob1new}, specifically, 
				\begin{gather*}
					a_{i2}y_{i2}^*+b_{i2}-x_{i}^*+\nu_{i2}^*=0, \qquad  x_i^*\le \bar{\rho}.
				\end{gather*}
			The first equality yields to $x^*_i > b_{i2}$ which contradicts $  x_i^*\le \bar{\rho}$, since $b_{i2}> \bar{\rho}$. Therefore, $y_{i2}^*=0$.
				\item The proof is similar to that of the previous statement.
			\end{enumerate}			
		\end{proof}}
		\begin{proof}[Proof of Lemma~\ref{lem:trivialsol}]
			\hfill
			\begin{enumerate}[label=\Roman*:]
				\item Since $b_i > \bar{\rho}$,  for any feasible $x_i$ such that $\bar{\rho} \ge x_i \ge \ubar{\rho}$, we can write $b_i > \bar{\rho} \ge {x_i} \ge  \ubar{\rho}$.  Therefore, $x_i < b_i$. Then, based on the objective function \eqref{prob1a} and the constraint \eqref{prob1d} , we can conclude that $x_i^* \in [ \ubar{\rho},  \bar{\rho}]$, $y_i^*=0$, $\mu_i^*=b_i-x_i^*$ and  $\nu_i^*=0$.
				\item Since $\ubar{\rho} >a_im_i+b_i $, for any feasible $x_i$ such that $\bar{\rho} \ge x_i \ge \ubar{\rho}$, we can write $ \bar{\rho} \ge {x_i} \ge  \ubar{\rho} >a_im_i+b_i $.  Therefore, $x_i > a_im_i+b_i$. Then, based on  \eqref{prob1d}, we can conclude that $x_i^* = \ubar{\rho}$, $y_i^*=m_i$, $\mu_i^*=0$ and $\nu_i^*=x_i^*-a_im_i-b_i$.
			\end{enumerate}
		\end{proof}		
		\begin{proof}[Proof of Lemma~\ref{lem:convex}]
			We consider two cases. Note that since $m_i>0$, $\mu_i$ and $\nu_i$ cannot be nonzero at the same time.
			\begin{description}
				\item[Case 1)] Suppose, for the sake of contradiction, $x_i^*$, $y_i^*$, $\mu_i^*$ and $\nu_i^*$ be the only optimal solution of the problem \eqref{prob1} and $\mu_i^*>0$ and $\nu_i^*=0$. Since $x_i^*$, $y_i^*$, $\mu_i^*$ and $\nu_i^*$ are an optimal solution and hence are feasible, they should satisfy the constraints of the problem \eqref{prob1}, .i.e. 
				\begin{gather*}
					\ubar{\rho} \le x_i^* \le \bar{\rho}, \\
					y_i^*+\sum_{j \in N \atop j\neq i}y_j^* \le f,\\ 
					y_i^*=0, \ \mu_i^*=b_i-x_i^*, \ \nu_i^*=0
				\end{gather*}
				We define $\bar{x}_i=x_i^*+\mu_i^*=b_i$ and consequently  we have
				\begin{gather*}
					\bar{y_i}=0, \ \bar{\mu}_i=0, \ \bar{\nu}_i=0.
				\end{gather*}
				Since $b_i \le \bar{\rho}$, $\bar{x}_i > x_i^*$ and $\bar{y}_i=y_i^*=0$, we can conclude that $\bar{x}_i$, $\bar{y}_i$, $\bar{\mu}_i$ and $\bar{\nu}_i$ are feasible for \eqref{prob1} and $\phi(x_i^*,y_i^*)= \phi (\bar{x}_i,\bar{y}_i)$ which is a contradiction.
				\item[Case 2)] Suppose, for the sake of contradiction, $x_i^*$, $y_i^*$, $\mu_i^*$ and $\nu_i^*$ be an optimal solution of the problem \eqref{prob1} and $\mu_i^*=0$ and $\nu_i^*>0$. Since $x_i^*$, $y_i^*$, $\mu_i^*$ and $\nu_i^*$ are an optimal solution and hence are feasible, they should satisfy the constraints of the problem \eqref{prob1}, .i.e.,
				\begin{gather*}
					\ubar{\rho} \le x_i^* \le \bar{\rho}, \\
					y_i^*+\sum_{j  \in N \atop j\neq i}y_j^* \le f,\\ 
					y_i^*=m_i, \ \mu_i^*=0, \ \nu_i^*=x_i^*-a_im_i-b_i.
				\end{gather*}
				We define $\bar{x}_i=x_i^*-\nu_i^*=a_im_i+b_i$ and consequently  we have
				\begin{gather*}
					\bar{y_i}=m_i, \ \bar{\mu}_i=0, \ \bar{\nu}_i=0.
				\end{gather*}
				Since $a_im_i+b_i \ge \ubar{\rho}$, $\bar{x}_i < x_i^*$ and $\bar{y}_i=y_i^*=m_i>0$, we can conclude that $\bar{x}_i$, $\bar{y}_i$, $\bar{\mu}_i$ and $\bar{\nu}_i$ are feasible for \eqref{prob1} and  $\phi (\bar{x}_i,\bar{y}_i) < \phi(x_i^*,y_i^*)$ which is  a contradiction.
			\end{description}
			By means of these two cases, we show that there exists an optimal solution $x_i^*$, $y_i^*$, $\mu_i^*=0$ and $\nu_i^*=0$ for all $i\in N$ for the problem \eqref{prob1}. As a result, we can relax $\mu_i$ and $\nu_i$ to zero in \eqref{prob1} for all $i \in N$ and thus the same $x_i^*$ and $y_i^*$ can be found as the solution of \eqref{prob2}.
		\end{proof}		
		\begin{proof}[Proof of Lemma~\ref{lem:bneg}]
			\hfill 
			\begin{enumerate}[label=\Roman*:]
				\item \textit{Problem PP:}  Let $x_i, y_i, \mu_i$ and $\nu_i$ be any feasible solution for \eqref{probpp}. Since $x_i$ is nonnegative and $b_i$ is negative, we have $x_i-b_i>0$  and $\mu_i =0$. If $\nu_i=0$, then $y_i=\frac{x_i-b_i}{a_i}>0$ and if $\nu_i>0$, then $y_i=m_i>0$. Therefore, $y_i$ is positive if $b_i$ is negative.\\
				\textit{Problem UP:} The proof is similar to the previous case.
				\item \textit{Problem PP:} It follows from Lemma~\ref{lem:trivialsol} item I.\\
				\textit{Problem UP:} The proof is similar to the proof of Lemma~\ref{lem:trivialsol} item I.
			\end{enumerate}	
		\end{proof}
		\begin{proof}[Proof of Lemma~\ref{lem:pp}]
			Since $0 \le b_i \le  \bar{\rho}$ for all $i \in N$,	based on Lemma~\ref{lem:convex}, $\mu_i^*=\nu_i^*=0$ for all $i \in N$ in \eqref{probpp}. Suppose, for the sake of contradiction, there exists $j \in N$ such that $y_j^*=0$ and $x_j^*=b_j$. We can leave out the constraints corresponding to the index $j$ from the problem \eqref{probpp}. As a result,  the following optimization problem has the same optimizer:
			\begin{subequations} \nonumber
				\begin{alignat}{2} 
					&  \underset{x,y}{\mathrm{min}} \quad  &&\ (x_j-p) (\frac{x_j-b_j}{a_j})+\sum_{i\in N\atop i\neq j} (x_i-p) (\frac{x_i-b_i}{a_i})+pf \\
					& \mathrm{subject \ to}\quad   &&0 \le x_i \le \bar\rho, \qquad \forall i\in N\setminus\{j\}, \\ 
					&&& \frac{x_i-b_i}{a_i} \ge 0, \qquad \forall i\in N\setminus\{j\}, \\
					&&& \frac{x_i-b_i}{a_i} \le m_i \qquad \forall i\in N\setminus\{j\}, \\
					&&& \sum_{i\in N\atop i\neq j} \frac{x_i-b_i}{a_i} \le f, 
				\end{alignat}
			\end{subequations} 
			where $\frac{x_i-b_i}{a_i} =y_i$.   This optimization problem is convex. Therefore, the optimizer of this problem  satisfies the KKT conditions. We write the KKT conditions of this problem for the index $j$ as follows:
			\begin{equation*}
				x_j^*=\frac{b_j+p}{2}.
			\end{equation*}
			Having $x_j^*=b_j$, we conclude $b_j=p$
			which is a contradiction since $b_i <p$ for all $i \in N$.
		\end{proof}
		\begin{proof}[Proof of Lemma~\ref{lem:up}]
			We define index sets $\gamma_1=\{i \in \gamma \mid \nu_i^*=0\}$ and $\gamma_2=\{i \in \gamma \mid \nu_i^*>0\}$. 
			Based on the sets $\gamma_1$, $\gamma_2$, $\bar{\gamma}$ and the complementary constraints of  the problem \eqref{probup}, the following can be concluded.
			\begin{equation*}
				\begin{aligned}
					&y_i^*>0, \quad &&\mu_i^*=0, \quad &\nu_i^*= 0,\qquad &\forall \ i \in \gamma_1,\\
					&y_i^*=m_i, \quad &&\mu_i^*=0, \quad &\nu_i^*> 0,\qquad &\forall \  i \in \gamma_2,\\
					&y_i^*=0, \quad &&\mu_i^*\ge 0, \quad &\nu_i^*= 0,\qquad &\forall \  i \in \bar{\gamma}
				\end{aligned}
			\end{equation*}
			As a result, the next optimization problem has the same  optimizer as \eqref{probup}. 
			\begin{subequations} \label{probupprime}
				\begin{alignat}{2} 
					&  \underset{x,y,\mu.\nu}{\mathrm{min}} \quad  &&\ (x-p)(\sum_{i\in \gamma_1} y_i+\sum_{i\in \gamma_2} m_i)+pf \label{probupprime0}\\
					& \mathrm{subject \ to}\quad  && y_i=\frac{x-b_i}{a_i} \ge 0 \qquad \forall i\in \gamma_1, \label{probupprime1}\\
					&&& \nu_i=x-a_im_i-b_i \ge 0, \qquad \forall i\in \gamma_2, \label{probupprime2}\\
					&&& \mu_i=b_i-x \ge 0, \qquad \forall i\in \bar{\gamma}, \label{probupprime3}\\
					&&& \sum_{i\in \gamma_1}y_i \le f- \sum_{i\in \gamma_2}m_i \label{probupprime4}
				\end{alignat}
			\end{subequations}
			Let $\lambda_{1i}$, $\lambda_{2i}$, $\lambda_{3i}$ and $\lambda_{4}$  be the dual variables for the constraints \eqref{probupprime1}-\eqref{probupprime4}, respectively. Since this problem is convex, the KKT conditions are necessary and sufficient for the optimizer of this problem. Note that since $y_i^*>0$, $\lambda_{1i}^*=0$ for all $i\in \gamma_1$ and since $v_i^*>0$, $\lambda_{2i}^*=0$ for all $i\in \gamma_2$.
			\begin{gather} \label{KKTpp1}
				x^*=\frac{p}{2}+\frac{\chi}{2}-\omega\\ \label{KKTpp2}
				x^* > b_i, \qquad \forall i \in \gamma_1,\\ \label{KKTpp3}
				x^* > a_im_i+b_i,\qquad \forall i \in \gamma_2,\\\label{KKTpp4}
				b_i \ge x_i \perp \lambda_{3i}^*\ge 0, \qquad \forall i \in \bar{\gamma},\\\label{KKTpp5}
				f- \sum_{i\in \gamma_2}m_i \ge \sum_{i\in \gamma_1}\frac{x-b_i}{a_i} \perp \lambda_{4}^*\ge 0 
			\end{gather}
			where 
			\begin{gather*}\label{eq:chi} 
				\chi=\frac{\displaystyle \sum_{i_1\in \gamma_1} (b_{i_1} \prod_{{i_2}\in \gamma_1 \atop i_2\neq i_1}a_{i_2}) }{\displaystyle \sum_{i_1\in \gamma_1}\prod_{i_2\in \gamma_1 \atop i_2\neq i_1}a_{i_2}},\\ \label{eq:omega}
				\omega=\frac{\displaystyle \prod_{i_1\in\gamma_1} a_{i_1}}{\displaystyle2\sum_{i_1\in \gamma_1}\prod_{i_2\in \gamma_1 \atop i_2\neq i_1}a_{i_2}}(\sum_{i_1 \in \gamma_2}m_{i_1}+\sum_{i_1\in \bar\gamma}\lambda_{3{i_1}}+\lambda_4\sum_{i_1\in \gamma_1}\frac{1}{a_{i_1}}).
			\end{gather*}
			Then, we have the following results.
			\begin{itemize}
				\item Let $j \in \gamma_1$ and $k \in \gamma_1 \setminus \{j\}$. Since $\omega$ is positive, we can conclude  from \eqref{KKTpp1} and \eqref{KKTpp2} 
				\begin{gather*} 
					x^*=\frac{p}{2}+\frac{\chi}{2}-\omega > b_j \quad \Longrightarrow \quad p>2b_j-\chi,\\
					x^*=\frac{p}{2}+\frac{\chi}{2}-\omega > b_k \quad \Longrightarrow \quad p>2b_k-\chi.
				\end{gather*}
				We multiply $p>2b_j-\chi$ by
				\[\sum_{i_1\in \gamma_1} ( \prod_{i_2\in \gamma_1 \atop i_2\neq i_1}a_{i_2}) + \prod_{i_2\in \gamma_1 \atop i_2\neq j}a_{i_2},\]
				and we multiply  $p>2b_k-\chi$ for all $k \in \gamma_1 \setminus \{j\}$ by 
				\[ \prod_{i_2\in \gamma_1 \atop i_2\neq k}a_{i_2}.\] 
				By adding these inequalities together, we have $p > b_j$.
				\item Let $j \in \gamma_2$. Due to \eqref{KKTpp1} and \eqref{KKTpp3}, we have 
				\begin{equation*}
					x^*=\frac{p}{2}+\frac{\chi}{2}-\omega >a_jm_j+b_j>b_j \quad \Longrightarrow \quad \chi+p > 2b_j.
				\end{equation*}
				Since $\chi$ is a weighted average of  the elements of the set $\{b_i \mid i \in \gamma_1 \}$,  there exists $k \in \gamma_1$ such that $b_k \ge \chi$. Also, since $k \in \gamma_1$, we have $p > b_k$.
				Therefore, 
				\begin{equation*}
					\begin{cases}
						\chi+p > 2b_j, \\
						b_k \ge \chi,\\
						p > b_k, 
					\end{cases} \quad \Longrightarrow \quad p > b_j.
				\end{equation*}
			\end{itemize} 
		\end{proof}
	\end{document}